\documentclass[11pt]{amsart}
\usepackage{amsmath,amssymb,amsthm,upref,graphicx,mathrsfs}

\usepackage{color,xcolor}
\usepackage[
  colorlinks=true,
  linkcolor=blue,
  citecolor=blue,
  urlcolor=blue]{hyperref}

\numberwithin{equation}{section}

\newtheorem{thm}{Theorem}[section]

\newtheorem{lem}[thm]{Lemma}

\numberwithin{equation}{section}

\begin{document}

\leftline{ \scriptsize}

\vspace{1.3 cm}
\title
{A class of normal dilation matrices affirming the Marcus-de Oliveira conjecture}
\author{Kijti Rodtes}
\thanks{{\scriptsize
\newline Keywords: Normal dilation, Normal matrices, Marcus de Oliveira Conjecture\\ MSC(2010): 15A15; 15A60; 15A86}}
\hskip -0.4 true cm

\maketitle


\begin{abstract}  In this article, we prove a class of normal dilation matrices affirming the Marcus-de Oliveira conjecture.
\end{abstract}

\vskip 0.2 true cm


\pagestyle{myheadings}
\markboth{\rightline {\scriptsize Kijti Rodtes}}
         {\leftline{\scriptsize }}
\bigskip
\bigskip


\vskip 0.4 true cm

Throughout, $n$ will denote a positive integer.  The determinant conjecture of Marcus and de Oliveira states that the determinant of the sum of two $n$ by $n$ normal matrices $A$ and $B$ belongs to the convex hull of the $n!$ $\sigma$-points, $z_\sigma:=\prod _{i=1}^n (a_i+b_{\sigma(i)})$, indexed by $\sigma\in S_n$, where $a_i$'s and $b_j$'s are eigenvalues of $A$ and $B$, respectively (see \cite{Marcus},\cite{Oliveira},\cite{zhan}).  We briefly write as $(A,B)\in MOC$ if the pair of normal matrices $A,B$ affirms the Marcus and de Oliveira conjecture, i.e., $$ \det(A+B)\in co( \{z_\sigma | \sigma \in S_n \}).$$  

In \cite{Fiedler}, Fiedler showed that, for two hermitian matrices $A,B$ $$\Delta(A,B):=\{\det(A+UBU^*)| U\in U_n(\mathbb{C})\}$$ is a line segment with $\sigma$-points as endpoints, where $U_n(\mathbb{C})$ denotes the set of all unitary matrices of dimension $n \times n$.  This result, in fact, motivates the conjecture.  As a consequence of Fiedler's result, $(A,B)\in MOC$ for any pair of skew-hermitian matrices $A,B$. 

In \cite{BEKOPRO}, N. Bebiano, A. Kovacec, and J.da Providencia provided that if $A$ is positive definite and $B$ a non-real scalar multiple of a hermitian matrix, then $(A,B)\in MOC$.  They also obtained that if eigenvalues of $A$ are pairwise distinct complex numbers lying on a line $l$ and all eigenvalues of $B$ lie on a parallel to $l$, then $(A,B)\in MOC$. S.W. Drury showed that $(A,B)\in MOC$ for the case that $A$ is hermitian and $B$ is non-real scalar multiple of a hermitian matrix (essentially hermitian matrix) in \cite{Dru1} and the case that $A=sU$ and $B=tV$ for $s,t\in \mathbb{C}$ and $U,V\in U_n(\mathbb{C})$ in \cite{Dru2}.   

It is also known that, for normal matrices $A,B\in M_n(\mathbb{C})$ (the set of all $n\times n$ matrices over $\mathbb{C}$), $(A,B)\in MOC$: if $\det(A+B)=0$ (\cite{Dru3}); if the point $z_\sigma$ lie all on a straight line (\cite{MV}); if $n=2,3$ (\cite{Oliveira,BMP}); if $A$ or $B$ has only two distinct eigenvalues, one of them simple, (\cite{Oliveira}).  However, it seems that there is no new affirmative class of normal matrices to this conjecture after the year 2007.

Let $X$ be a square $n\times n$ complex matrix and $s$ be a complex number.  It is a direct calculation to see that 
$$ N(X,s):=\left(
\begin{array}{cc}
X & (X-sI)^* \\
(X-sI)^* & X \\
\end{array}
\right)  $$
is a normal matrix of size $2n\times 2n$ and thus it is a normal dilation of $X$. We will see (in the proof of the main result) that the eigenvalues of $N(X,s)$ lie on both real and imaginary axis and thus this matrix need not be essentially hermitian or a scalar multiple of a unitary matrix.  In this short note, we show that:
\begin{thm}\label{mainresult}
	Let $X,Y\in M_n(\mathbb{C})$ and $s,t\in \mathbb{C}$.  Then $(N(X,s),N(Y,t))\in MOC$.
\end{thm}
Note that if $A\in M_n(\mathbb{C})$ is normal then $UAU^*$ is also normal for any $U\in U_n(\mathbb{C})$.  Then $V N(X,s)V^*$ is also a normal dilation of $X$ for any $V\in U_{2n}(\mathbb{C})$.  Moreover, since the conjecture is invariant under simultaneous unitary similarity, we also deduce from Theorem \ref{mainresult} that $(VN(X,s)V^*,VN(Y,t)V^*)\in MOC$ for any $V\in U_{2n}(\mathbb{C})$.

To prove the main result, we will use the following lemmas.

\begin{lem}\label{tool1} Let  $A,B\in M_n(\mathbb{C})$ and $C,D\in M_m(\mathbb{C})$ be normal.  If $(A,B)\in MOC$ and $(C,D)\in MOC$, then $(A\oplus C, B\oplus D)\in MOC$.
\end{lem}
\begin{proof} Suppose that $\{a_i \,|\, 1\leq i\leq n \}$, $\{b_i \,|\, 1\leq i\leq n \}$, $\{c_i \,|\, 1\leq i\leq m \}$ and $\{d_i\,|\, 1\leq i\leq m \}$ are ordered set of the eigenvalues of $A,B,C$ and $D$, respectively.  Denote $e_i:=a_i$, $f_i:=b_i$ for $i=1,\dots,n$ and $e_{n+j}=c_j$, $f_{n+j}=d_j$ for $j=1,\dots,m$.  Then, $\{ e_i\,|\, 1\leq i \leq n+m\}$ and $\{ f_i\,|\, 1\leq i \leq n+m\}$ are ordered set of the eigenvalues of $A\oplus C$ and $B\oplus D$, respectively.  For each $\sigma\in S_n, \pi\in S_m$ and $\theta\in S_{n+m}$, denote $z_\sigma, v_\pi$ and $w_\theta$ the product $\prod_{i=1}^n(a_i+b_{\sigma(i)})$, $\prod_{i=1}^m(c_i+d_{\pi(i)})$ and $\prod_{i=1}^{n+m}(e_i+f_{\theta(i)})$, respectively.  Suppose that $(A,B)\in MOC$ and $(C,D)\in MOC$, then $$\det(A+B)=\sum_{\sigma\in S_n}t_\sigma z_\sigma  \hbox{ and } \det(C+D)=\sum_{\pi\in S_m}s_\pi v_\pi,$$ where $t_\sigma,s_\pi \in [0,1]$ such that $\sum_{\sigma\in S_n} t_\sigma=1$ and $\sum_{\sigma\in S_m} s_\pi=1$.  Note that
\begin{eqnarray*}
  \det(A\oplus C+B\oplus D) &=& \det((A+B)\oplus(C+D) )\\
   &=& \det(A+B)\cdot \det(C+D) \\
   &=&(\sum_{\sigma\in S_n}t_\sigma z_\sigma)(\sum_{\pi\in S_m}s_\pi v_\pi)\\
   &=& \sum_{\sigma\in S_n,\pi\in S_m}(t_\sigma s_\pi) (z_\sigma v_\pi).
\end{eqnarray*}
For each $\sigma\in S_n$ and $\pi\in S_m$, define a permutation $\theta(\sigma,\pi)\in S_{n+m}$ by $$\theta(\sigma,\pi):=\left(
                                                                             \begin{array}{cccccc}
                                                                               1 & \cdots & n & n+1 & \cdots & n+m \\
                                                                               \sigma(1) & \cdots & \sigma(n) & n+\pi(1) & \cdots & n+\pi(m) \\
                                                                             \end{array}
                                                                           \right)
$$
Then $w_{\theta(\sigma,\pi)}=z_\sigma v_\pi$.  Since, for each $\sigma\in S_n$ and $\pi\in S_m$, $t_\sigma s_\pi\in[0,1]$ and  $$\sum_{\sigma\in S_n,\pi\in S_m}(t_\sigma s_\pi)=(\sum_{\sigma\in S_n} t_\sigma)(\sum_{\sigma\in S_m} s_\pi)=(1)(1)=1,$$
we conclude that  $$\det(A\oplus C+B\oplus D)\in co\{ w_{\theta(\sigma,\pi)} \,|\, \sigma\in S_n,\pi\in S_m\}\subseteq co\{w_\theta \,|\, \theta\in S_{n+m} \}.$$
Hence $(A\oplus C, B\oplus D)\in MOC$.
\end{proof}
To be a self contained material, we record a result of S.W. Drury. 
\begin{thm}\label{key1}\cite{Dru4}
	Let $A$ and $B$ be hermitian matrices with the given eigenvalues $(a_1,\dots, a_n)$ and $(b_1,\dots,b_n)$ respectively.  Let $(t_1,\dots,t_n)$ be the eigenvalues of $A+B$.  Then
	$$\prod_{j=1}^n (\lambda+t_j)\in co\{\prod_{j=1}^n(\lambda+a_j+b_{\sigma(j)}) | \sigma \in S_n \},  $$
	where $co$ denotes the convex hull in the space of polynomials and $\lambda$ is an indeterminate.
\end{thm}
As a corollary of the above theorem, we have that:
\begin{lem} \label{tool2}
	Let $X,Y\in M_n(\mathbb{C})$ and $\alpha,\beta\in \mathbb{C}$. Then $(X-X^*+\alpha I_n,Y-Y^*+\beta I_n) \in MOC$  and $(X+X^*+\alpha I_n,Y+Y^*+\beta I_n) \in MOC$.
\end{lem}
\begin{proof}
	Since $X+X^*$ and $Y+Y^*$ are hermitian, by Theorem \ref{key1}, we deduce directly that $(X+X^*+\alpha I_n,Y+Y^*+\beta I_n) \in MOC$.  Since $X-X^*$ and $Y-Y^*$ are skew-hermitian, $i(X-X^*)$ and $i(Y-Y^*)$ are hermitian.  Again, by Theorem \ref{key1}, $(X-X^*+\alpha I_n,Y-Y^*+\beta I_n) \in MOC$.

\end{proof}

\begin{proof}

\textit{(Theorem \ref{mainresult})} Let $U$ be the block matrix in $M_{2n}(\mathbb{C})$ defined by $$U:=\frac{1}{\sqrt{2}}\left(
\begin{array}{cc}
I_n & I_n \\
-I_n & I_n \\
\end{array}
\right).$$
It is a direct computation to see that $U$ is a unitary matrix and $$ U^*\left(\begin{array}{cc}
M & N \\
N & M \\
\end{array}\right)
U= (M-N)\oplus (M+N), $$
for any $M,N\in M_n(\mathbb{C})$.  Let $A:=X-X^*+(\overline{s}) I_n$, $B:=Y-Y^*+\overline{t} I_n$, $C:=X+X^*-(\overline{s}) I_n$, and  $D:=Y
+Y^*-\overline{t} I_n$. By Lemma \ref{tool2}, the pair of normal matrices $(A,B)$ and $(C,D)$ satisfy the conjecture. Hence, by Lemma \ref{tool1}, $(A\oplus C, B \oplus D)\in MOC$.  Therefore,
$$(N(X,s),N(Y,t))=(U (A\oplus C)U^*,U (B\oplus D)U^* )\in MOC, $$
which completes the proof.
\end{proof}

\section*{Acknowledgments}

The author would like to thank Prof Tin Yau Tam for bringing this topic to the author.  He would like to thank the referee(s) for valuable comments to improve the paper.  He also would like to thank Naresuan University for the financial support on the project number R2563C006.

\bigskip

\address \textbf{Kijti Rodtes} \\

{ Department of Mathematics, Faculty of Science, \\ Naresuan University, Phitsanulok 65000, Thailand}\\
\email{kijtir@nu.ac.th, \quad \quad}\\

\end{document}